\documentclass[12pt,a4paper]{article}

\usepackage{amsmath}
\usepackage{amssymb}
\usepackage{amsthm}
\usepackage[small,margin=3mm]{caption}
\usepackage{color}
\usepackage{dsfont}
\usepackage{graphicx}
\usepackage[hidelinks]{hyperref}
\usepackage[latin1]{inputenc}
\usepackage{mathrsfs}
\usepackage{mathtools}
\usepackage{psfrag}
\usepackage[active]{srcltx}
\usepackage{url}
\usepackage{verbatim}

\newcommand{\dd}{{\mathrm d}}
\newcommand{\E}{\mathbb{E}}

\newcommand{\Z}{\mathbb{Z}}
\newcommand{\ZZ}{\widetilde\Z}
\newcommand{\R}{\mathbb{R}}
\newcommand{\Pb}{\mathbb{P}}

\theoremstyle{plain}
\newtheorem{theorem}{Theorem}

\newtheorem{lemma}{Lemma}
\newtheorem{corollary}{Corollary}

\theoremstyle{definition}

\newtheorem*{definition*}{Definition}
\theoremstyle{remark}

\newtheorem*{claim*}{Claim}

\newtheorem*{example*}{Example}

\begin{document}

\title{Site Percolation on a Disordered \\ Triangulation of the Square Lattice}
\author{Leonardo T. Rolla}
% \date{March 01, 2016}
\maketitle

% {\raggedleft
% \emph{to the 70th birthday of \\ Charles M. Newman}
% \par}
% \vspace{1.0em}

\begin{abstract}
In
this paper we consider independent site percolation in a triangulation of $\R^2$ given by adding $\sqrt{2}$-long diagonals to the usual graph $\Z^2$.
We conjecture that $p_c=\frac{1}{2}$ for any such graph, and prove it for almost every such graph.
\end{abstract}

\section{The model}

Let $\ZZ^2$ denote the set of $1\times 1$ squares in $\R^2$ having all its corners in $\Z^2$, and define the diagonal configuration space by
\[
\Omega=\big\{
\vcenter{\hbox{\includegraphics[page=1,height=1.4em]{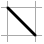}}}
,
\vcenter{\hbox{\includegraphics[page=2,height=1.4em]{figures/diagonals}}}
\big\}^{{\ZZ^2}}
.
\]
Let $\Sigma$ denote the color configuration space
\[
\Sigma=\{ 
\vcenter{\hbox{\includegraphics[page=1,height=1em]{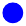}}}
,
\vcenter{\hbox{\includegraphics[page=2,height=1em]{figures/colors}}}
\}^{\Z^2}
.
\]
Examples of a diagonal configuration $\omega$ and a color configuration $\sigma$ are
\[
\vcenter{\hbox{\includegraphics[page=3,width=.3\textwidth]{figures/sample}}}
\qquad
\text{and}
\qquad
\vcenter{\hbox{\includegraphics[page=2,width=.3\textwidth]{figures/sample}}}
.
\]

Let $P_p$ denote the probability measure on $\Sigma$ given by
\[
P_p(\sigma_x=\vcenter{\hbox{\includegraphics[page=2,height=.8em]{figures/colors}}})=\boldsymbol{\color{red}{p}}
\qquad
\text{and}
\qquad
P_p(\sigma_x=\vcenter{\hbox{\includegraphics[page=1,height=.8em]{figures/colors}}})=\boldsymbol{\color{blue}{1-p}}
,
\]
independently over $x\in \Z^2$, and let $P_p^\omega$ denote the law of the percolation process on the graph obtained by adding the diagonals in $\omega$ to the usual graph $\Z^2$.
In general, the resulting graph does not have any symmetry, but still there is a critical parameter $p_c(\omega)$ at which the probability of having an infinite red cluster jumps from $0$ to $1$ (this is a tail event).

% For a given diagonal configuration $\omega \in \Omega$,

Since this graph is a triangulation, site percolation is self-dual, that is, the only way to prevent a given red connection is with a transversal blue connection and vice-versa.
This is illustrated by
\[
{\hbox{\includegraphics[page=1,width=.35\textwidth]{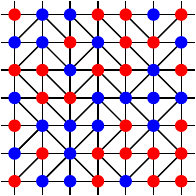}}}
,
\]
where existence of a left-right red crossing prevents a top-bottom blue crossing, and a left-right blue crossing prevents a top-bottom red crossing.

\section{Results}

Because of self-duality, the obvious conjecture%
\footnote{If $p_c$ were smaller, there would be both an infinite blue and an infinite red cluster for any $p\in(p_c,1-p_c)$, and if $p_c$ were larger, there would be both infinitely many red and blue circuits surrounding the origin for any $p\in(1-p_c,p_c)$.}
is that $p_c(\omega)=\frac{1}{2}$.
In this paper we show that this is true if the diagonal configuration $\omega$ is obtained by tossing a fair coin for each square $z \in \ZZ^2$.

\medskip

Let $Q$ denote the probability on $\Omega$
given by
\[
Q\big(\omega_z=\vcenter{\hbox{\includegraphics[page=1,height=1.4em]{figures/diagonals}}}\big)
=
Q\big(\omega_z=\vcenter{\hbox{\includegraphics[page=2,height=1.4em]{figures/diagonals}}}\big)
=
\tfrac{1}{2}
,
\]
independently over different squares $z \in \ZZ^2$.
Our sample space will be $\Omega \times \Sigma$, so the process described above is governed by the ``quenched measure''
\[
P_p^\omega=\delta_\omega\times P_p
.
\]
% be the distribution of the quenched configuration on $\Omega \times \Sigma$ given $\omega$,
In this paper we will consider the ``annealed measure'' given by
\[
\Pb_p=Q\times P_p=\int_\Omega P_p^\omega \, \dd Q(\omega)
.
\]

\begin{theorem}
% [Annealed phase transition]
\label{theopc}
For the annealed process,
\[
\Pb_p(\mbox{there is an infinite red cluster}) =
\begin{cases}
0,& p\leqslant\frac 12, \\
1,& p>\frac12.
\end{cases}
\]
In particular, by Fubini-Tonelli,
\(
p_c(\omega)=\frac{1}{2} \text{ for almost every } \omega
.
\)
\end{theorem}

\bigskip

The proof is elementary and self-contained.
The main step is to establish the analogous of Harris-FKG inequality for the annealed measure.

% \medskip

We note that not all events which we would normally call ``increasing'' are positively correlated, because two such events may have conflicting requirements for the diagonals.
To overcome this, we define \emph{robust increasing} events, and prove positive correlations for this type of event only.

We then show with pictures how Smirnov's proof of RSW estimates~\cite{Russo78,SeymourWelsh78} and Kesten's proof of logarithmic expected number of pivotal sites~\cite{Kesten80} can both be adapted to the present model.
For completeness, we give the proof of Theorem~\ref{theopc} using Russo's formula.

% After that, we consider an exploration procedure that reveals each diagonal configuration as it enters the respective square, and reveals site colors as they appear at the opposite corner of the currently-visited triangle.

% First we define a notion of increasing event that includes both site colors and diagonal orientations.
% Not all events we would like to call ``increasing'' 
% With this notion 
% is, so that some types of 

% The annealed law lacks positive correlations, as the occurrence of certain ``increasing'' events may give information about the diagonals that turn out to be negative for some other.
% However we cannot prove Russo-Seymour-Welsh (RSW) estimates for an arbitrary diagonal configuration, and some quenched properties can actually be shown through the annealed case.
% The key observation is that positive correlations are still true for a certain class of events.

% Theorem~\ref{theopc} says that $p_c(\omega) = \frac{1}{2}$ for $Q$-a.e.\ $\omega$, and we believe it is true for every $\omega\in\Omega$.

% The same proof would work if one could show that the estimate in Lemma~\ref{lemmarsw} holds not only in average, but uniformly in $\omega$.

\section{Robust increasing events}

An \emph{observable} is a measurable function $f:\Omega\times\Sigma\to\R$.
\begin{definition*}
We say that an observable $f$ is \emph{increasing in $\sigma$} if, for each pair $(\omega,\sigma)$, switching the color of any site $x$ from $\vcenter{\hbox{\includegraphics[page=1,height=.8em]{figures/colors}}}$ to $\vcenter{\hbox{\includegraphics[page=2,height=.8em]{figures/colors}}}$ increases (i.e., does not decrease) the value of $f$.
\end{definition*}

In order to discuss monotonicity with respect to the diagonal configuration $\omega$, we take into account the color configuration $\sigma$ to see whether it is
$\vcenter{\hbox{\includegraphics[page=1,height=1.2em]{figures/diagonals}}}$
or
$\vcenter{\hbox{\includegraphics[page=2,height=1.2em]{figures/diagonals}}}$
who
favors
the
$\vcenter{\hbox{\includegraphics[page=2,height=.8em]{figures/colors}}}$'s
more
than
the
$\vcenter{\hbox{\includegraphics[page=1,height=.8em]{figures/colors}}}$'s,
or the other way around.

Given a color configuration $\sigma$, each square $z\in\ZZ^2$ will be classified as having one of three types, depending on the colors of its four corners.

The first type consists of configurations whose symmetries make it impossible to decide whether
$\vcenter{\hbox{\includegraphics[page=2,height=.8em]{figures/colors}}}$'s
and
$\vcenter{\hbox{\includegraphics[page=1,height=.8em]{figures/colors}}}$'s
would
prefer
$\vcenter{\hbox{\includegraphics[page=1,height=1.2em]{figures/diagonals}}}$
or
$\vcenter{\hbox{\includegraphics[page=2,height=1.2em]{figures/diagonals}}}$.
\par
--
type N:
$
\vcenter{\hbox{\includegraphics[page=4,height=1.88em]{figures/types}}},\
\vcenter{\hbox{\includegraphics[page=13,height=1.88em]{figures/types}}},\
\vcenter{\hbox{\includegraphics[page=7,height=1.88em]{figures/types}}},\
\vcenter{\hbox{\includegraphics[page=10,height=1.88em]{figures/types}}},\
\vcenter{\hbox{\includegraphics[page=1,height=1.88em]{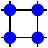}}},\
\vcenter{\hbox{\includegraphics[page=16,height=1.88em]{figures/types}}}
.
$

\begin{definition*}
We say that $f$ is \emph{robust} if flipping the diagonal at any square $z$ of type N does not change the value of $f$.
\end{definition*}

Squares that are not of type N will have type A or B depending on whether 
it is the
$\vcenter{\hbox{\includegraphics[page=2,height=.8em]{figures/colors}}}$'s
or
$\vcenter{\hbox{\includegraphics[page=1,height=.8em]{figures/colors}}}$'s
who prefer
$\vcenter{\hbox{\includegraphics[page=2,height=1.2em]{figures/diagonals}}}$
over
$\vcenter{\hbox{\includegraphics[page=1,height=1.2em]{figures/diagonals}}}$.

% \\
% \hspace*{1em}
\par
--
type A:
$
\vcenter{\hbox{\includegraphics[page=11,height=1.88em]{figures/types}}},\
\vcenter{\hbox{\includegraphics[page=3,height=1.88em]{figures/types}}},\
\vcenter{\hbox{\includegraphics[page=9,height=1.88em]{figures/types}}},\
\vcenter{\hbox{\includegraphics[page=12,height=1.88em]{figures/types}}},\
\vcenter{\hbox{\includegraphics[page=15,height=1.88em]{figures/types}}}
. 
$
% \\
% \hspace*{1em}
\par
--
type B:
$
\vcenter{\hbox{\includegraphics[page=6,height=1.88em]{figures/types}}},\
\vcenter{\hbox{\includegraphics[page=14,height=1.88em]{figures/types}}},\
\vcenter{\hbox{\includegraphics[page=8,height=1.88em]{figures/types}}},\
\vcenter{\hbox{\includegraphics[page=5,height=1.88em]{figures/types}}},\
\vcenter{\hbox{\includegraphics[page=2,height=1.88em]{figures/types}}}
.
$

\begin{definition*}
We say that a robust observable $f$ is \emph{increasing in $\omega$} if, for each pair $(\omega,\sigma)$, flipping the diagonal of any square $z$ from $\vcenter{\hbox{\includegraphics[page=1,height=1.2em]{figures/diagonals}}}$ to $\vcenter{\hbox{\includegraphics[page=2,height=1.2em]{figures/diagonals}}}$ will
increase $f$ for $z$ of type A and decrease $f$ for $z$ of type B.
An event $\mathcal{A}$ is called robust, increasing in $\sigma$ and $\omega$ if its indicator function is so.
\end{definition*}

% Given $\omega\in\Omega$ and $z\in\ZZ^2$, define $\omega^z\in\Omega$ by $\omega^z_z=S_z\backslash\omega_z$ and $\omega^z_w=\omega_w$ for $w\neq z$, i.e. the configuration $\omega$ with the diagonal around $z$ flipped.

% for all $\omega,\sigma,z$, $f(\omega^z,\sigma)\not=f(\omega,\sigma)$ implies that $\sigma_z\in\mathcal D_z$.

%  $f(\omega,\sigma)\geqslant f(\omega^z,\sigma)$ whenever $\omega_z=\sigma_z$.

% We say that $f$ is \emph{robust} if,
% ...
% In words, $f$ is robust if it only depends on the diagonal of a given square if it may be important to make a connection inside that square.

% Finally we say that $f$ is \emph{increasing in $\omega$} if
% ...
% In words, that means that $f$ increases when a diagonal is flipped in the favor of the black sites.

\medskip

We mention that some events that would normally be called ``increasing'' are not robust.
For example, in a $3 \times 3$ square (what we call $3 \times 3$ contains $16$ sites), existence of a red path of length $3$ connecting the top-left and bottom-right corners is not robust.
Moreover, this event is not positively-correlated with existence of a red path of length $3$ connecting the top-right and bottom-left corners: they are in fact mutually exclusive.
The above events are not robust because they have requirements for diagonals even when the containing squares are of type N.
In the same direction, events requiring existence of disjoint paths are not robust in general.

On the other hand, and that is enough for our needs, for any sets $A,B \subseteq \Z^2$ and any domain $\mathcal{D}$ consisting of a collection of closed squares of $\ZZ^2$, the event ``$A$ is connected to $B$ by a red path in $\mathcal{D}$'' is both robust and increasing.

\medskip

% We now proceed to the proofs.

% ??????
% Examples or robust observables are: $[0\leftrightarrow\infty]$, $[x\leftrightarrow y]$, $[x\leftrightarrow y\mbox{ inside a box}]$, left-right crossing of a box.
% The following are not robust: existence of disjoint paths, length of the shortest crossing, $[x\leftrightarrow y\mbox{ in }B]$ where $B$ is a domain having narrow diagonal connections.

\begin{lemma}
[Harris-FKG]
\label{lemmafkg}
Let $f$ and $g$ be robust non-negative observables, increasing in $\sigma$ and $\omega$.
Then \[\Pb_p(fg)\geqslant\Pb_p(f)\Pb_p(g).\]
\end{lemma}

\begin{proof}
% [Proof of Lemma~\ref{lemmafkg}]
% For an observable $h$, let $h_\sigma=h(\cdot,\sigma)$ for each fixed $\sigma$.
% Observe that, for each fixed $\sigma$, there is a natural order on each of the coordinates of $\Omega$ such that $h_\sigma$ is increasing with respect to the corresponding coordinate-wise partial order whenever $h$ is robust and increasing in $\omega$.
% Since $Q$ is a product measure this implies the Harris-FKG inequality~[ref] for projections of such observables.
% For an observable $h$, let $h_\sigma=h(\cdot,\sigma)$ for each fixed $\sigma$.
% Observe that, for each fixed $\sigma$, since $h$ is robust and increasing in $\omega$ there is an obvious order for each of the coordinates of $\Omega$ such that $h_\sigma$ is increasing with respect to the corresponding coordinate-wise partial order.
% Since $Q$ is a product measure this implies the Harris-FKG inequality~[ref] for projections of such observables.
Let $\sigma\in\Sigma$ be fixed.
For an observable $h$, consider the projection $h_\sigma:\Omega\to\R$ given by $h_\sigma=h(\cdot,\sigma)$.
Observe that, if $h$ is robust and increasing in $\omega$, then $h_\sigma$ depends only on $\omega_z$ for $z$ outside the set $N_\sigma \subseteq \ZZ^2$ of squares of type N.
Moreover, there is a natural partial order on $\{\vcenter{\hbox{\includegraphics[page=1,height=1.2em]{figures/diagonals}}},\vcenter{\hbox{\includegraphics[page=2,height=1.2em]{figures/diagonals}}}\}^{N_\sigma^c}$ under which $h_\sigma$ is an increasing function ($\vcenter{\hbox{\includegraphics[page=1,height=1.2em]{figures/diagonals}}}\leq \vcenter{\hbox{\includegraphics[page=2,height=1.2em]{figures/diagonals}}}$ for $z\in N_\sigma^c$ of type A and $\vcenter{\hbox{\includegraphics[page=1,height=1.2em]{figures/diagonals}}}\geq \vcenter{\hbox{\includegraphics[page=2,height=1.2em]{figures/diagonals}}}$ for $z\in N_\sigma^c$ of type B).
Since $Q$ induces a product measure on $\{\vcenter{\hbox{\includegraphics[page=1,height=1.2em]{figures/diagonals}}},\vcenter{\hbox{\includegraphics[page=2,height=1.2em]{figures/diagonals}}}\}^{N_\sigma^c}$, projections of this type satisfy the Harris-FKG inequality with respect to $Q$.
Therefore, if $f$ and $g$ are non-negative, robust and increasing in $\sigma$ and $\omega$,
\[
  \Pb_p(fg)= P_p[Q(f_\sigma g_\sigma)] \geqslant P_p[Q(f_\sigma)Q(g_\sigma)] \geqslant P_p[Q(f_\sigma)]P_p[Q(g_\sigma)] = \Pb_p(f)\Pb_p(g).
\]
%   \int fg \Pb_p(d\omega, d\sigma) = \int [ \int fg Q(d \omega) P_p(\dd \sigma) ] \geqslant P_p[Q(f)Q(g)] \geqslant P_p[Q(f)]P_p[Q(g)] = \Pb_p(f)\Pb_p(g)
We have used Fubini-Tonelli theorem for the equalities.
The first inequality follows from the above observation and the second inequality follows from the standard Harris-FKG inequality, since $f$ and $g$ are increasing in $\sigma$.
\end{proof}

% Using the above lemma, it is possible to adapt one of the existing proofs of RSW estimates [ref].

\section{Proof of sharp percolation threshold}

\begin{lemma}
[Russo-Seymour-Welsh]
\label{lemmarsw}
In any $2n \times n$ rectangle,
\[
\Pb_{\frac{1}{2}}
\Bigg(
\vcenter{\hbox{\includegraphics[page=1,width=.4\textwidth]{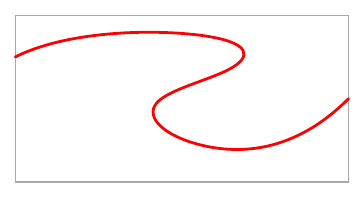}}}
\Bigg)
\geqslant
\frac{1}{16}
.
\]
\end{lemma}

In the proof we consider an exploration that progressively reveals the color of some sites and the position of some diagonals. Below we show an exploration starts from the top-left corner and targets the bottom-right corner of a rectangle. When the exploration enters a triangle by crossing one of its sides, it looks at the color of the opposite corner in order to decide on where to exit the triangle. When it enters a square, it first reveals the position of the diagonal on that square first.
\begin{equation}
\label{eq:exploration}
\vcenter{\hbox{\includegraphics[page=1,width=.4\textwidth]{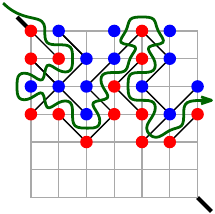}}}
\end{equation}
In this procedure, the exploration will leave the rectangle through the right side before the bottom side if there is a left-right red connection, and the bottom side before the right otherwise.

\begin{proof}
[Proof of Lemma~\ref{lemmarsw}]
We use $\Pb$ for $\Pb_{\frac{1}{2}}$.
By Harris-FKG we have
\[
\Pb
\Bigg(
\vcenter{\hbox{\includegraphics[page=1,width=7em]{figures/rswproof}}}
\Bigg)
\geqslant
\Pb
\Bigg(
\vcenter{\hbox{\includegraphics[page=2,width=7em]{figures/rswproof}}}
\Bigg)
\times
\Pb
\Bigg(
\vcenter{\hbox{\includegraphics[page=3,width=7em]{figures/rswproof}}}
\Bigg)
,
\]
whence by symmetry it suffices to show that
\begin{equation}
\label{eq:weakconnection}
\Pb
\Bigg(
\vcenter{\hbox{\includegraphics[page=2,width=9em]{figures/rswproof}}}
\Bigg)
\geqslant
\frac{1}{4}
.
\end{equation}
We first try to ``bend'' the left-side boldface region by starting an exploration path in the left-side square as shown in~(\ref{eq:exploration}). With probability $\frac{1}{2}$ we succeed bending the boldface region until the middle of the rectangle, revealing some diagonals and some blue and red sites like this:
\[
{\hbox{\includegraphics[page=4,width=.501\textwidth]{figures/rswproof}}}
.
\]
Given a partial configuration such as above, the event in~(\ref{eq:weakconnection}) is equivalent to a red connection between the two boldface regions.
So we need to show that the conditional probability of such red connection is at least $\frac{1}{2}$.
But such connection is certainly implied by red path connecting two smaller boldface regions contained in a smaller grayed zone given by
\[
{\hbox{\includegraphics[page=5,width=.501\textwidth]{figures/rswproof}}}
.
\]
Now notice that none of the sites and diagonals revealed so far can interfere with this event, except for some red sites lying on the boldface region.
The fact that these sites are red can only help, and the conditional probability of the latter event given that they are red is bounded from below by the probability of the crossing
\[
\vcenter{\hbox{\includegraphics[page=6,width=.501\textwidth]{figures/rswproof}}}
\]
without any conditioning.
Finally, the complementary of the latter event is
\[
\includegraphics[page=7,width=.501\textwidth]{figures/rswproof},
\]
which by symmetry has the same probability, concluding the proof.
\end{proof}

\pagebreak[3]
As usual, Lemma~\ref{lemmarsw} has the following immediate corollaries.

\begin{corollary}
\label{cor:longish}
There is $\delta>0$ such that, in any $8n \times n$ rectangle,
\[
\Pb_{\frac{1}{2}}
\bigg(
\vcenter{\hbox{\includegraphics[page=1,width=.5\textwidth]{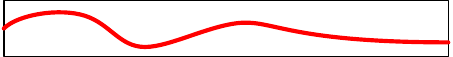}}}
\bigg)
\geqslant
\delta
.
\]
\end{corollary}
\begin{corollary}
\label{cor:circuit}
There is $\delta>0$ such that, in any pair of co-centered squares of size $4n \times 4n$ and $6n \times 6n$,
\[
\Pb_{\frac{1}{2}}
\Bigg(
\vcenter{\hbox{\includegraphics[page=1,width=.3\textwidth]{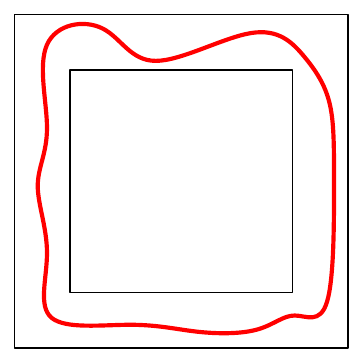}}}
\Bigg)
\geqslant
\delta
,
\]
\end{corollary}

% \begin{corollary}
% There is $\alpha>0$ such that, in any $n \times n$ square,
% \[
% \Pb_{\frac{1}{2}}
% \Bigg(
% \vcenter{\hbox{\includegraphics[page=2,width=.3\textwidth]{figures/circuit}}}
% \Bigg)
% \leqslant
% \frac{1}{n^{\alpha}}
% .
% \]
% \end{corollary}

The last piece in the proof is the following.

\begin{lemma}
[Kesten]
\label{lemma:kesten}
There is $\beta>0$ such that, for any $p\geqslant \frac{1}{2}$, in any $2n \times n$ rectangle, the expected number of pivotal sites satisfies
% \footnote{remove boldface}
\[
\E_{p}
\Bigg(
\vcenter{\hbox{\includegraphics[page=1,width=.3\textwidth]{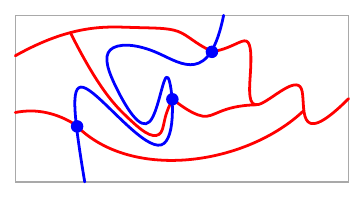}}}
\Bigg)
\geqslant
% \beta n^\beta
\beta \log_2 n
\times
\Pb_{p}
\Bigg( 
\vcenter{\hbox{\includegraphics[page=2,width=.3\textwidth]{figures/pivotals}}}
\Bigg)
.
\]
\end{lemma}

\begin{proof}
% [Proof of Lemma~\ref{lemma:kesten}]
It suffices to show that
\[
\E_{p}
\Bigg(
\vcenter{\hbox{\includegraphics[page=1,width=.351\textwidth]{figures/pivotals}}}
\Bigg|
\vcenter{\hbox{\includegraphics[page=2,width=.351\textwidth]{figures/pivotals}}}
\Bigg)
\geqslant
% \beta n^\beta
\beta \log_2 n
.
\]
We will determine occurrence of the top-bottom blue crossing using an exploration path that starts at the top-left corner and ends at either the bottom or the right side, as below.
\[
\includegraphics[page=1,width=.711\textwidth]{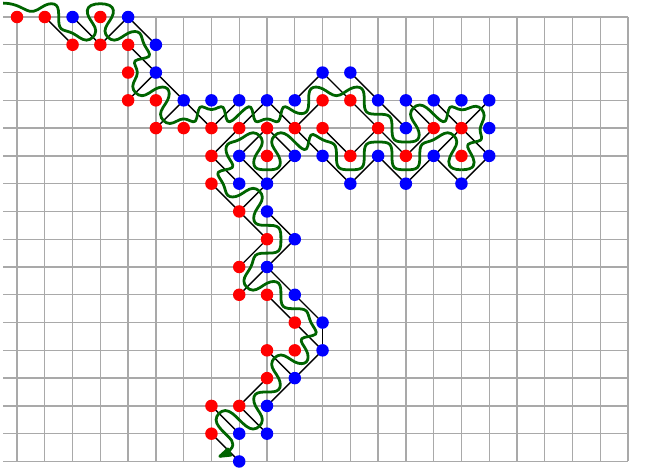}
\]
Existence of a top-bottom blue crossing is equivalent to the exploration finding the bottom side before the right side.
We want to show that, if it such crossing occurs, then the conditional expectation of the number of pivotal sites given the colors and diagonals revealed in this exploration is greater than $\beta \log_2 n$ for some constant $\beta$.

On the above event, there is a self-avoiding blue path that joins the top and bottom sides of the rectangle.
What we do now is a little overkilling, but it avoids the hassle of considering all corner cases related to the diagonals.
Let us first inflate this self-avoiding blue path to make it squared where it would otherwise use a diagonal, as below.
\[
\includegraphics[page=2,width=.711\textwidth]{figures/pivotalsproof}
\]
Notice that there are two types of sites in this squared path.
The first type consists of blue sites which are adjacent to a red site which is in turn connected to the left side of the rectangle by a red path.
The second type consists of sites whose color has not yet been revealed, but which are adjacent to a site of the first type.

Now, in order to find pivotal sites, we consider the domain consisting of the squares that can be reached from the right side of the big rectangle without crossing the squared path, colored in light-gray below.
\[
\includegraphics[page=3,width=.711\textwidth]{figures/pivotalsproof}
\]
We will ``re-sample'' the entire process on this domain. Of course we cannot do that, so the red sites that we find on the squared path that turn out to have been previously sampled as blue will remain blue in the end.

We first look for a connection from the right side of the big rectangle to the squared-path that stays in a strip of width $\frac{n}{4}$ using an exploration path as below.
\[
\includegraphics[page=4,width=.711\textwidth]{figures/pivotalsproof}
\]
The last red site was drawn smaller to remind us that it may be a site which we already revealed to be blue.
By Corollary~\ref{cor:longish}, the probability of finding such a connection is greater than $\delta$.

In case such path is found, we now draw disjoint $4\cdot2^k$-sided and $6\cdot2^k$-sided squares centered at the point just found, and consider the intersection of the corresponding annuli with the light-gray region in the second previous picture. There are at least $\delta \log_2 n$ such regions that do not go lower than the bottom side of the big rectangle, minus the squares explored in the previous step, as shown below.
\[
\includegraphics[page=5,width=.711\textwidth]{figures/pivotalsproof}
\]

Conditioned on the above picture, by Corollary~\ref{cor:circuit} each of these ``tunnels'' will contain a red connection with probability at least $\delta$.
The conditional expectation for the number of such tunnels that actually contain such connection, given that the previous connection has been found, is thus greater than $\delta^2 \log_2 n$.
Therefore, the conditional expectation given existence of a top-bottom blue crossing is greater than $\delta^3 \log_2 n$.

In the example below, two such tunnels ended up providing one red site (drawn smaller), and one of them did not.
\[
\includegraphics[page=6,width=.711\textwidth]{figures/pivotalsproof}
\]

We combine the configuration discovered at this stage with the one previously removed. At this point, a small red site will become a true red site if it had not been revealed in the first exploration, and will be reverted to blue in case it had. The result is highlighted by a light-gray disk below.
\[
\includegraphics[page=7,width=.711\textwidth]{figures/pivotalsproof}
% \qedhere
\]
To conclude the proof, notice that each of these highlighted sites is either a pivotal site in case it is blue, or is preceded by a pivotal site in the squared curve in case it is red.
\end{proof}

\bigskip

\begin{proof}
[Proof of Theorem~\ref{theopc}]
% Write $p=\frac{1}{2}+\epsilon$.
Absence of percolation at $p=\frac{1}{2}$ follows from Corollary~\ref{cor:circuit} as usual.
Define the event
\[
\mathcal{A}_n
=
\vcenter{\hbox{\includegraphics[page=2,width=.3\textwidth]{figures/pivotals}}}
\]
in a $2n \times n$ rectangle.
Using Russo's formula and Lemma~\ref{lemma:kesten} we get
\[
\tfrac{\dd}{\dd p} \Pb_p(\mathcal{A}_n)
=
-
% \tfrac{\dd}{\dd p} \Pb_p(\mathcal{A}^c)
\E_p\big[\,\text{pivotal sites}\,\big]
\leqslant
-
\beta \log_2 n \, \Pb_p(\mathcal{A}_n)
,
\]
which gives
\[
\log \Pb_{\frac{1}{2}+\epsilon}(\mathcal{A}_n) \leqslant -\epsilon \beta \log_2 n
,
\]
and thus, on an $2^k \times 2^{k+1}$ rectangle,
\[
\Pb_{\frac{1}{2}+\epsilon}
\Bigg(
\vcenter{\hbox{\includegraphics[page=1,width=.3\textwidth]{figures/rswproof}}}
\Bigg)
\geqslant
1-e^{\displaystyle -\epsilon \beta k}
.
\]
To conclude the proof, we arrange $2^k \times 2^{k+1}$ rectangles as
\[
{\hbox{\includegraphics[page=1,width=.5\textwidth]{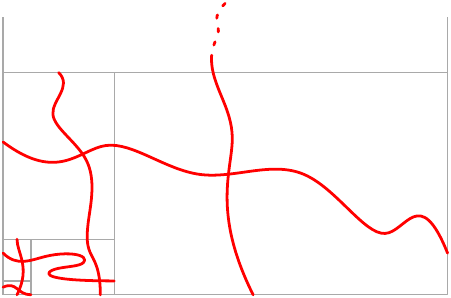}}}
% \qedhere
\]
and deduce the second part of the theorem using Harris-FKG inequality.
\end{proof}

\section{Extensions and alternative approaches}

Recent work on Bernoulli percolation in $\Z^d$ and Voronoi percolation in $\R^2$ provided new insights for the study of sharp phase transitions~\cite{Duminil-CopinTassion16,Tassion16}. It would be interesting to apply those ideas to disordered triangulations of the square lattice, and obtain alternative proofs or extend the present results to a more general setting.

% The understanding of percolation phase transition has been greatly simplified thanks to recent work by Duminil-Copin and Tassion, and it is possible that their approach leads to new insight for the study of disordered triangulations of the square lattice.
% In particular, combining it with techniques introduced for Voronoi percolation might allow for extensions of the results obtained here, for instance having biased coins for the diagonal configurations.

\section*{Acknowledgement}

I would like to thank Wendelin Werner for suggesting this problem to me back in 2009, and for inspiring discussions.
I also thank him for pointing out the possible connections discussed above.

% I should also mention that this research was done with support of Fondation Sciences Mathématiques de Paris, PIP 11220130100521CO, PICT-2015-3154, PICT-2013-2137, PICT-2012-2744, Conicet grant 45955 and MinCyT-BR-13/14.

{
\renewcommand{\baselinestretch}{1}
\setlength{\parskip}{0cm}
\small
\par
\bibliographystyle{abbrv}
\bibliography{bib/leo}
\par
}

\end{document}